\documentclass[12pt]{amsart}

\usepackage{amsmath}
\usepackage{amsfonts}
\usepackage{amssymb}
\usepackage{epsfig}
\usepackage{epstopdf}
\usepackage{comment}
\usepackage{color, url}
\usepackage[footnote]{fixme}

\definecolor{blue}{RGB}{0,0,255}
\definecolor{green}{RGB}{50,150,50}
\definecolor{red}{RGB}{255,0,0}

\newtheorem{theorem}{Theorem}[section]
  \newtheorem{lemma}[theorem]{Lemma}

   \newtheorem{corollary}[theorem]{Corollary}

  \theoremstyle{definition}
  
  \newtheorem{example}[theorem]{Example}
   \newtheorem{remark}[theorem]{Remark}

\newcommand{\ideal}[1]{\langle #1 \rangle}

\newcommand{\CC}{\mathbb C}
\newcommand{\RR}{\mathbb R}
\newcommand{\QQ}{\mathbb Q}
\newcommand{\ZZ}{\mathbb Z}

\newcommand{\ba}{\mathbf a}

\newcommand{\be}{\mathbf e}
\newcommand{\bx}{\mathbf x}
\newcommand{\cT}{\mathcal T}

\newcommand{\cA}{\mathcal A}

\DeclareMathOperator{\init}{in}
\DeclareMathOperator{\tinit}{t-in}

\DeclareMathOperator{\trop}{\cT}
\DeclareMathOperator{\V}{\mathcal{V}}

\newcommand{\mult}[2]{\operatorname{mult}_{#2}({#1})}

\date{\today}
 
\title[]{Computing tropical curves\\via homotopy continuation}
\author{Anders Jensen}
\address{ Institut for Matematik, Aarhus University, Aarhus, Denmark }
\email {jensen@imf.au.dk}
\author{Anton Leykin} 
\address{ School of Mathematics, Georgia Institute of Technology, Atlanta GA, USA}
\email{leykin@math.gatech.edu}
\author{Josephine Yu}
\address{ School of Mathematics, Georgia Institute of Technology, Atlanta GA, USA}
\email{jyu@math.gatech.edu}
\thanks{}

\date{\today}

\begin{document}
\maketitle
 
\begin{abstract} 
Exploiting a connection between amoebas and tropical curves, we devise a method for computing tropical curves using numerical algebraic geometry and give an implementation. As an application, we use this technique to compute Newton polygons of $A$-polynomials of knots. 
\end{abstract}
 
\section{Introduction}

We present a construction of the tropicalization of a complex curve using numerical methods.  Our procedures produce numerical data such as floating point approximations of points on the curve and then translate them into discrete information such as primitive integer vectors along the rays of the desired tropical curve and their multiplicities.  

The connection between numerical homotopy continuation and tropical geometry first appeared in the work of Huber and Sturmfels on polyhedral homotopies~\cite{HuberSturmfels}.  Recent works in this direction include numerical recovery of truncated Puiseux series for curves and surfaces~\cite{Adrovic-Verschelde:Puiseux-for-surfaces,Adrovic-Verschelde:exploiting-symmetry-cyclic} and a method to compute tropical hypersurfaces numerically~\cite{Hauenstein-Sottile:NP-and-witness-sets}.

As tropical varieties are subfans of Gr\"obner fans, it is natural that the current methods for computing tropical varieties rely on Gr\"obner basis computations.  The most general method is implemented in the software Gfan~\cite{Gfan} and described in~\cite{BJSST}. We present in \S\ref{section:main} a new method for computing tropical curves numerically.  

In the fan traversal methods such as those used in Gfan, one has to compute tropical varieties locally at a codimension one cone, and this computation can be reduced to the case of a curve. The case of curves is an important building block for computing tropical varieties in general. An alternative algorithm presented in~\cite{chanphd} computes tropical curves by computing a set of two variable elimination ideals and thereby the projections of the tropical curve to a set of coordinate 2-planes. From this combinatorial data the tropical curve is reconstructed. However, this technique still relies heavily on Gr\"obner bases. If the numerical method of ~\cite{Hauenstein-Sottile:NP-and-witness-sets} is employed to compute projected planar tropical curves, it is likely to suffer from a ``high-degree-low-dimension curse'': the projection of a variety cut out by low-degree polynomials to a 2-plane may be given by a single (unknown) polynomial of high degree in two variables. 

 In the framework of numerical algebraic geometry (numerical AG)~\cite{SVW9,Sommese-Wampler-book-05}, we do not compute {\em any} Gr\"obner bases and use homotopy continuation algorithms instead.  For implementations we use three software packages developed for numerical AG: PHCpack~\cite{PHCpack}, Bertini~\cite{Bertini-book}, and NAG4M2~\cite{Leykin:NAG4M2}. Particular features of all three are necessary: see how these are combined within the Macaulay2~\cite{M2www} environment in \S\ref{section:implementation}.

While certification --- giving the numerically computed results the grade of a proof --- is currently not possible in the general case, we develop hybrid symbolic-numerical validation procedures that shall provide confidence in the obtained results in many cases. For further discussion see~\S\ref{subsec:certification}.

\medskip
\noindent
{\bf Convention:}
We use the {\em max convention} for tropical geometry:  the initial form of a polynomial contains the {\em maximal} degree terms, the normal vectors of Newton polytopes are outward-pointing, and the degree of ``$t$'' in Puiseux series and in $t$-initial forms is $-1$.

\medskip
\noindent
{\bf Acknowledgments:}  We would like to thank Henry Duong, who successfully studied several small examples of the problem for his REU project at Georgia Tech in 2011; Robert Krone, who participated in our early discussions; and Stavros Garoufalidis, who consulted us on examples coming from the knot theory (see~\S\ref{section:knots}). AJ was supported by the Danish Council for Independent Research, Natural Sciences (FNU). AL was partially supported by NSF-DMS grant \#~1151297, and JY by \#~1101289.

\section{Computing tropical curves}\label{section:main}

Let $I$ be an ideal in a polynomial ring $\CC[x_1,\dots,x_n]$.  The tropical variety of $I$ is the following subfan of the Gr\"obner fan of $I$
$$
\trop(I) = \{ w \in \RR^n : \init_w(I) \text{ contains no monomials}\}.
$$
Throughout this paper we assume that the variety of $I$ in the algebraic torus~$(\CC^*)^n$ is equidimenstional of dimension one, although it may have components of different dimensions in the boundary of the torus.  Although in theory it is more natural to consider $I$ as ideal in the Laurent polynomial ring $\CC[x_1^{\pm1},\dots,x_n^{\pm 1}]$, we choose to write about polynomials because they better reflect the way computations are done.

We will describe a procedure for computing the tropical variety numerically, as follows.  First compute the degree of the curve in the torus by slicing with a generic hyperplane.  Then proceed with the following steps.
\begin{enumerate}
\item Find some possible rays in the tropical variety by sampling points along the tentacles of the amoeba.
\item Compute for each found possible ray its multiplicity in the tropical curve.  The ray has multiplicity zero if it is not in the tropical curve.
\item Check whether the rays, with multiplicities, make up a tropical curve with correct degree.  If not, go back to Step (1).
\end{enumerate}
We will describe each of the three steps in detail in the following three subsections.

Our procedure produces the correct tropical curve assuming that the numerical computations are reasonably reliable. To find all the rays of the tropical curve, it suffices, in theory, to take two generic parallel slices of the amoeba or any $n+1$ hyperplane slices whose origin facing normal vectors positively span $\RR^n$ --- assuming that slices are taken far away from the origin. Moreover, Corollary~\ref{cor:degree} below limits the possible candidate rays of the tropical curve to a finite set.  Therefore the Step~(1) will not go on forever.  If we are conservative with numerical procedures in Step~(2), then we will not pick up false rays for our tropical curve.  Using Lemma~\ref{lem:degree} below we keep track of how far we are from a complete curve with the right degree and this tells us when to terminate. 

\subsection{Finding rays in tropical curves by slicing amoebas}
\label{section:find}

We will use amoebas to find candidates for vectors in the tropical curve.

Let $V(I)$ be the variety of $I$ in $(\CC^*)^n$. Consider the logarithm map 
$$\log : (\CC^*)^n \rightarrow \RR, ~~ (z_1,\dots,z_n) \mapsto (\log|z_1|, \dots, \log|z_n|).$$ The image of $V(I)$ under the map $\log$ is called the {\em amoeba} of $V(I)$ and is denoted $\cA(I)$ \cite{GKZ}. See Figure~\ref{fig:amoeba}.
The tropical variety is the limit:
$$
\trop(I) = \lim_{t \rightarrow \infty} \frac{1}{t} \cA(I).
$$
In the book by Maclagan and Sturmfels \cite{MaclaganSturmfels}, this is called the {\em Bergman construction} of tropical varieties from amoebas.  When $V(I)$ is a curve, the tropical variety consists of a finite set of rays in directions of the {\em tentacles} of $\cA(I)$.

\begin{figure}
\includegraphics[scale=0.8]{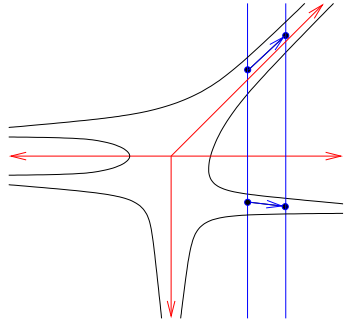}
\caption{Slice the amoeba with two  hyperplanes to find approximate directions along the tentacles.}
\label{fig:amoeba}
\end{figure}

Using numerical AG, we can compute numerical approximations of points in a zero-dimensional variety. We compute points along tentacles in the amoeba and find integer vectors along the tentacles, as follows. 

Let $H$ be a (usual) hyperplane $\{X \in \RR^n : X \cdot \ba = c\}$ where $\ba$ is an integer vector and $c$ is a real number.  Then $H$ is the amoeba of the hypersurface defined by a binomial $\bx^\ba = C$ where $C$ is a complex number with $\log(|C|) = c$.  We use homotopy continuation to compute the solutions $V(I_{\ba,C})$ where
\begin{equation}\label{equation:I_aC}
I_{\ba,C} := I + \langle  \bx^\ba - C \rangle.
\end{equation} 
The logarithms of the solutions lie in the intersection $\cA(I) \cap H$.  We then choose another complex number $C'$ with $|C'| > |C|$ and track the original solutions to the parallel slice $V(I_{\ba,C'})$. To this end we arrange a homotopy $C(s)$ in $s\in[0,1]$ such that $|C(s)|$ increases monotonously, $C(0)=C$, and $C(1)=C'$.  This way, if a point from the first slice lies in a tentacle, then it is guaranteed to stay in the same tentacle while moving toward the second slice.  

\smallskip

Now we hope to find integer vectors close to the difference vectors (see Figure~\ref{fig:amoeba}). In this step, we rely heuristically on either lattice reduction algorithms such as LLL or continued fraction techniques.  

We briefly describe the LLL trick.  Let $\bx = (x_1,\dots,x_n)$ be a vector in $\RR^n$, and suppose we wish to find a vector with small integer entries in approximately the same direction as $\bx$.  We first multiply $\bx$ with a large constant and round off to integers, so that $\bx$ is an integer vector with large entries.  Let $M$ be an $n \times (n-1)$ integer matrix whose columns span a lattice perpendicular to $\bx$.  Heuristically, if $x_k$ has the largest absolute value among all the $x_i$'s, then we take the columns of $M$ to be $x_i \be_k - x_k \be_i$ for $i \neq k$. Let $M' = \left[ M \,|\, I \,\right]$, an $n \times ((n-1)+n)$ matrix.  Apply the LLL lattice basis reduction algorithm to find a short vector $\bx'$ in the lattice spanned by the {\em rows} of $M'$. The first $n-1$ entries of $\bx'$ should be small.  This means that the vector of the last $n$ entries (which captures how the small vector was obtained from the rows of $M$) should be almost perpendicular to the columns of $M$; hence it is a small integer vector almost in the direction of $\bx$.

\begin{remark}\label{remark:C(s)}
In practice, ideally, we want to have a set $F$ of generators of $I$ that form a complete intersection and such that the polynomial system \[F_{\ba,C(s)} = (F,\bx^\ba - C(s))\] is a square system with regular solutions for $s\in[0,1]$.    

Assume that $F$ is indeed a regular sequence. Moreover, assume that $F_{\ba,C}$ has only regular (isolated) solutions for all but finitely many $C\in \CC$. Let $C(s) = (1+sA)C$ with positive $A\in\RR$. For a generic choice of $C$ one can prove that the corresponding homotopy is regular, i.e.\ the jacobian $\frac{\partial}{\partial \bx} F_{\ba,C(s)}$ does not vanish at points $V(F_{\ba,C(s)})$ for $s\in[0,1]$, which is essential for numerical path-tracking algorithms. 

Relying on the usual probabilistic algorithms  of numerical AG such as {\em squaring up} and {\em deflation}, one can bring the general case to the scenario above. For instance, given some generators $G$ of $I$, one can produce a regular sequence $F$ by taking $n-1$ generic linear combinations of $G$. This results in $V(F)$ that contains our curve $V(I)$ as a component. 
\qed
\end{remark}

We need to show that we do not miss any tentacles of the amoeba while taking a binomial slice.  

\begin{lemma}
Let $I$ be an ideal defining a curve in $(\CC^*)^n$.  Let $C$ be a positive real number. For any real number $\alpha$, the logarithm of the points in $\V(I + \langle \bx^\ba - C\, e^{\alpha i} \rangle)$ meets every connected component of the intersection of amoeba $\cA(I)$  with the hyperplane $H$ defined by $ X \cdot \ba = \log(C)$.
\end{lemma}

\begin{proof} 
As $\theta$ varies in the circle $\RR / 2\pi$, the image of  $\V(I + \langle \bx^\ba - C \, e^{\theta i} \rangle)$ under the logarithm map traces out all points in $\cA(I) \cap H$. Hence every point in $\cA(I) \cap H$ is connected to a point in the logarithm of $\V(I + \langle \bx^\ba - C \, e^{\alpha i} \rangle)$ via a path staying inside $\cA(I) \cap H$.
\end{proof}

We may not know whether the slicing hyperplane goes through the tentacles or the ``body'' of the amoeba, and we may not know if the amoeba tentacle is thin enough for our approximation to give the correct integer vector.  In practice, the procedure described above involves many heuristics and results in a set of candidate rays.  It remains to explain how to:
\begin{itemize}
\item verify that a given vector belongs to the tropical curve;
\item verify that a set of vectors is complete, i.e.\ that no rays are missing.
\end{itemize}

\subsection{Computing multiplicity numerically}
\label{section:check}
The key ingredient in both tasks above is computing the {\em multiplicity} of points in a tropical variety. Let $I$ be an ideal in $\CC[x_1,\dots,x_n]$, and let $\omega \in \RR^n$.  The multiplicity $\mult{\trop(I)}{\omega}$ of $\omega$ in $\trop(I)$ is the sum of multiplicities of monomial-free minimal associated primes of $\init_\omega(I)$~\cite{MaclaganSturmfels}.  If the point $\omega$ is not in the tropical variety $\trop(I)$, then the multiplicity is zero.  

The multiplicity is also the degree of the initial ideal after ``taking out the torus action'' as follows.  Let $\langle \init_\omega (I) \rangle$ denote the ideal in the Laurent polynomial ring $\CC[x_1^{\pm 1},\dots,x_n^{\pm 1}]$ generated by $\init_\omega (I)$. The Bieri--Groves theorem states that the dimension of the tropical variety $\trop(I)$ is equal to the Krull dimension of the ideal~$I$.  If $\omega$ is in the relative interior of a maximal Gr\"obner cone of $\trop(I)$, then $\init_\omega(I)$ is homogeneous with respect to gradings in a linear space $L$ of dimension $\dim(I)$.  The ideal $\langle \init_\omega (I) \rangle \cap \CC[L^\perp \cap \ZZ^n]$ is a zero dimensional ideal whose length (or degree) is equal to $\mult{\trop(I)}{\omega}$.  To compute the multiplicity symbolically, we can find a generating set of $\langle \init_\omega (I) \rangle$ consisting of Laurent polynomials with exponents lying in $L^\perp$.  After choosing a lattice basis for $L^\perp \cap \ZZ^n$, we can rewrite the generators as Laurent polynomials in $\dim(L)$ variables.  The desired multiplicity is the degree of the ideal they generate.  

With numerical AG, we do not compute the generators of $\init_\omega (I)$, so we must devise a new method. Our idea is to cut the variety down to zero dimension by binomials, and then to compute the number of Puiseux series solutions of a zero dimensional ideal, which also gives the multiplicity of the tropical variety.  We first need to introduce the notion of multiplicities for ideals over Puiseux series.

Let $K$ be the field of Puiseux series in $t$ with complex coefficients convergent in a punctured neighborhood of $0$ and $I \subset K[x_1,\dots,x_n]$.  Let $\omega \in \RR^n$.  Following~\cite{JMM} we define the {\em $t$-initial form} $\tinit_\omega(f)$ of a polynomial $f \in K[x_1,\dots,x_n]$ with respect to $\omega$ as follows.  First take the sum of maximal degree terms in $f$ where $t$ has degree $-1$ and $x_1,\dots,x_n$ have degree $\omega_1,\dots,\omega_n$ respectively, then substitute $1$ for all powers of $t$.  For example, $$\tinit_{(1,2)}((3t^{1/2}+1)xy + ty + 5 t^{-3/2} x) = (3t^{1/2}xy + 5 t^{-3/2}x)|_{t=1} = 3xy+5x.$$
The {\em $t$-initial ideal} $\tinit_\omega(I)$ of $I$ is the ideal in $\CC[x_1,\dots,x_n]$ generated by the $t$-initial forms of elements of $I$.  

The tropical variety of $I$ over $K$ is defined as the set of $\omega$ such that the $\tinit_\omega(I)$ contains no monomials, and the multiplicity of $\omega$ in $\trop(I)$ as the sum of multiplicities of the monomial-free minimal associated primes of $\tinit_\omega(I)$. If the ideal $I$ is generated by polynomials over $\CC$, then the $t$-initial ideal coincides with the usual initial ideal, and the tropical variety over $K$ is the same as that over $\CC$.

We will now look at how multiplicities change when we intersect a tropical curve and a (usual) hyperplane.  Let $I$ be an ideal in $\CC[x_1,\dots,x_n]$ defining a curve in $(\CC^*)^n$ as before, and let $H$ be a hyperplane defined by $ X \cdot v = c$ where $v$ is a primitive integral  vector and $c \in \QQ\setminus\{ 0\}$.  Then $\trop(I)$ is a tropical curve and $H$ is equal to $\trop(\bx^v - C)$ for any $C \in K$ with $\deg(C)=c$. (Note that the degree of $t$ is $-1$.) Since $\trop(I)$ consists of rays emanating from the origin and $H$ does not go through the origin, the intersection $\trop(I) \cap H$ is transverse, i.e.\ every intersection point lies in the relative interior of a ray in $T$.  
Then
$$\trop(I + \langle \bx^v - C \rangle) = \trop(I) \cap_{st} \trop(\langle \bx^v - C \rangle) = \trop(I) \cap H,$$
where $\cap_{st}$ denotes {\em stable intersection}, and the first equality holds since $t$ is transcendental over the original polynomials, making the coefficients of the binomial {\em generic}~\cite[Section~3]{stableIntersection}.
By the multiplicity formula for stable intersections, for a point $\omega$ in the intersection lying in a ray in primitive integral direction $r$, we have
$$\mult{\trop(I + \langle \bx^v - C \rangle)}{\omega} = |r \cdot v| \cdot \mult{\trop(I)}{\omega}.$$  Note that the $H$ has multiplicity $1$ as the tropical variety of $\bx^v - C$ because $v$ is primitive.


\medskip

Let us now go back to the  problem of computing $\mult{\trop(I)}{\omega}$ numerically. We will assume that $\omega\in\ZZ^n\setminus\{0\}$ is primitive. First we cut the variety down to zero dimension as follows. We choose $v\in\ZZ^n$ such that  $\omega \cdot v = -1$. This is possible since $\omega$ is primitive.  The tropical variety of the binomial $\bx^v + t$ is the hyperplane defined by $X \cdot\, v = -1$, which intersects the candidate ray transversely at $\omega$.  If $\omega \in \trop(I)$, then $\omega$ is also in $\trop(I + \langle \bx^v + t \rangle)$, and the multiplicity of $\omega$ in $\trop(I + \langle \bx^v + t \rangle)$ is equal to the multiplicity of $\omega$ in $\trop(I)$ because $|\omega \cdot v| = 1$.

\begin{figure}
\includegraphics[scale=1.2]{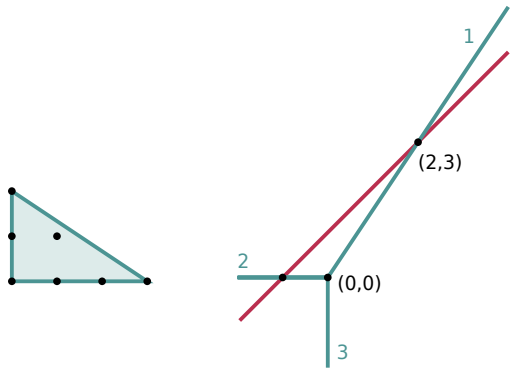}
\caption{The Newton polygon and the tropical curve of the polynomial $1+x^3+y^2$ are shown in blue.  The tropical curve of the binomial $ x + ty$ is shown is red.  It has slope $1$ and meets the blue curve with multiplicity $1$ at $(2,3)$.  See Example~\ref{ex:triangle}.}
\label{fig:triangle}
\end{figure}

\begin{example}
\label{ex:triangle}
Let $I = \langle 1 + x^3 + y^2 \rangle$.  Suppose we want to check whether the vector $\omega = (2,3)$ is in the tropical curve $\trop(I)$.  The vector $v = (1,-1)$ satisfies $\omega \cdot v = -1$.  Then $\omega \in \trop(I)$ if and only if $\omega \in \trop( \langle 1 + x^3 + y^2, x + ty \rangle)$. See Figure~\ref{fig:triangle}. \qed
\end{example}

We have now reduced the problem of computing multiplicities to zero dimensional ideals over $K$. The next statement follows from \cite[Proposition~3.4.8]{MaclaganSturmfels}.



\begin{theorem}[Fundamental Theorem of Tropical Geometry with multiplicities] 

For a zero dimensional ideal $J$ and any $\omega \in \trop(J)\cap\QQ^n$, the multiplicity of~$\omega$ in $\trop(J)$ is equal to the number of Puiseux series zeroes of $J$ with valuation~$\omega$, counted with multiplicity.
\label{thm:fundamental}
\end{theorem}





By substituting $x_i$ with $t^{-\omega_i} x_i$, we may assume that $\omega = 0$. Then we would like to check whether any zero in $(K^*)^n$ of $I$ has valuation $0$.   

\medskip
\noindent{\bf Example~\ref{ex:triangle} continued.} We now substitute $x$ with $x t^{-2}$ and $y$ with $y t^{-3}$ to obtain $\langle 1 + t^{-6} x^3 + t^{-6} y^2, t^{-2} x + t^{-2} y \rangle$. We would like to check whether the ideal has a zero in $(K^*)^2$ all of whose coordinates have non-zero constant terms. Note that as $t \rightarrow 0$, the dominant terms are $x^3 + y^2$ and $x+y$, which are the $t$-initial forms of the original polynomials.
\qed

\begin{lemma}\label{lemma:main}
Let $I \subset K[x_1,\dots,x_n]$ define a zero dimensional variety $V(I)\subset (K^*)^n$ with a generating set $F \subset \CC[t^{\pm 1},x_1,\dots,x_n]$.  For $a \neq 0$ in $\CC$, let $I_a\subset\CC[x_1,\dots,x_n]$ be the ideal obtained from $I$ by substituting $t$ with $a$.  As $a \rightarrow 0$, zeroes of $I_a$ either converge to the coordinate-wise constant terms of points in $V(I)$ or diverge to $\infty$.
\end{lemma}
\begin{proof}
Consider $t$ as a complex variable and consider the curve in $(\CC^*)^{n+1}$ defined by $I$.
By Puiseux's theorem, the curve is  parameterized by Puiseux series in $t$ locally near $t=0$.  Thus for sufficiently small $a$, the points along homotopy paths of zeroes of $I_a$ as $a \rightarrow 0$ are obtained by plugging in $a$ for $t$ in the Puiseux series solutions.  Either they diverge to $\infty$ when the degree is positive, or they converge to the constant term. (Recall that the degree of $t$ is $-1$.)  
\end{proof}

\noindent{\bf Example~\ref{ex:triangle} continued.} 
Following the 3 homotopy paths for zeroes of $J = \langle 1 + t^{-6} x^3 + t^{-6} y^2, t^{-2} x + t^{-2} y \rangle$ as $t \rightarrow 0$ finds us one point $(-1,1)$ in the torus, which is also a zero of the $t$-initial ideal $\tinit_0(J)$.  Its coordinates are leading coefficients of the Puiseux series solution of $J$ with degree $(0,0)$.  






\begin{remark}\label{remark:numerical-multiplicity}
Assume the case of complete intersection and regularity as in Remark~\ref{remark:C(s)}. The multiplicity of a ray spanned by $\omega$ equals the number of paths converging to points in the torus as $a\to 0$ in Lemma~\ref{lemma:main} as one takes the homotopy for a generic smooth path $a(s)$, $s\in[0,1]$, $a(1)=0$. For instance, $a(s)=A(1-s)$ for a generic $A\in\CC$ works: it avoids singularities with a possible exception of $s=1$.

In the general case, in addition to standard regularization techniques mentioned in Remark~\ref{remark:C(s)}, one can compute the multiplicities of the converging continuation paths numerically via Macaulay dual spaces (e.g.\ see \cite{Krone:NumericalHilbert} for the description of the method and a software implementation). These multiplicities are the multiplicities of the corresponding Puiseux series. \qed
\end{remark}

\subsection{Checking Completeness}
\label{section:degree}

After computing a collection of rays with multiplicities, we wish to decide if we have found all rays.  First we can check if the balancing condition is satisfied.  If it is, then we can check if the degree of the current curve agrees with the degree of the classical curve in $(\CC^*)^n$, which can be computed numerically by counting the number of solutions in the intersection of the curve with a generic hyperplane. To compute the degree of the  tropical curve found thus far, we can compute its stable intersection with a tropical hyperplane \cite[Section 3]{stableIntersection}.  This computation is very easy to do for curves, as shown by the following lemma. Recall that we are using the max-convention, so the tropical hyperplane has rays in direction $-e_1,\dots,-e_n$, and $(1,\dots,1)$ and contains the cones spanned by any $n$ of the rays.

\begin{lemma}[Degree of Tropical Curve]
\label{lem:degree}
Suppose a tropical curve in $\RR^n$ consists of rays in primitive integral directions $r_1, \dots, r_k$ with multiplicities $m_1,\dots,m_k$ respectively.  We can decompose each $r_i$ as a positive linear combination of $-e_0, -e_1,\dots,-e_n$, where $e_0 : = -e_1-\cdots-e_n$, such that not all of the $e_i$'s are used. The degree of the tropical curve is the number of each $e_i$ we get this way, counted with multiplicities. 
\end{lemma}

For example, consider rays in directions $(-1,2)$, $(-2,-3)$, and $(4,-1)$ with multiplicities $2$, $1$, and $1$ respectively. We can then decompose the rays as
\begin{align*}
(1,-2) & = 0\cdot(-1,0) + 3\cdot(0,-1) + 1\cdot(1,1)\\
(2,3) &= 1 \cdot (-1,0) + 0 \cdot(0,-1)+ 3 \cdot(1,1)\\
(-4,1) &= 5 \cdot(-1,0) + 0 \cdot(0,-1) + 1 \cdot(1,1).
\end{align*}
The degree of the curve is $6$. Note that $(-1,2)$ has multiplicity $2$.

\begin{proof}
Let $r$ be a primitive vector in the tropical curve with multiplicity $m$.
Since $-e_0,-e_1,\dots,-e_n$ positively span $\ZZ^n$, we can write $mr = a_0 (-e_0) + \cdots + a_n (-e_n)$ where $a_i$'s are non-negative integers and at least one of them is zero. WLOG, suppose $a_0 = 0$.  Consider the tropical cycle on rays $r, -e_1, \dots, -e_n$ with multiplicities $-m$, $a_1, \dots, a_n$ respectively. Its support lies entirely in the cone spanned by $-e_1, \dots, -e_n$. The stable intersection of this tropical cycle with a tropical hyperplane is zero because we can translate the cycle in direction $e_0$ and get an empty intersection with the tropical hyperplane.  Adding this tropical cycle to the original tropical curve does not change the stable intersection with a hyperplane but has the effect of replacing the ray $r$ with the collection of rays $-e_1, \dots, -e_n$, with multiplicities.  In this way, we can transform our tropical curve, while preserving the degree, to the curve consisting of rays $-e_0,\dots,-e_n$, each with multiplicity $d$. The degree of the new curve is the multiplicity of the origin in its stable intersection with the tropical hyperplane, which is ~$d$.
\end{proof}

Our implementation carries out the above computation as follows: Lift the found rays to $\RR^{n+1}$ by adding a zero coordinate.  For each ray, subtract the maximum coordinate value from all coordinates.  The sum of the new rays is a multiple of $(-1,-1,\dots,-1)$ in $\ZZ^{n+1}$ if and only if the rays satisfy the balancing condition.  This multiple is the degree of the tropical curve.  

In the example above, we first write the rays as $(1,-2,0)$, $(2,3,0)$, and $(-4,1,0)$ then as $(0,-3,-1)$, $(-1,0,-3)$, and $(-5,0,-1)$ by subtracting away the maximum coordinates. Their sum with multiplicities is $(-6,-6,-6)$, so the degree of the curve is $6$.

If the sum is not equal to $-\deg(I) \cdot (1,1,\cdots,1)$, then we know some rays are missing, and the difference gives us an idea of how far we are from the correct answer.

The following is an immediate consequence of Lemma~\ref{lem:degree}.
It limits the possible primitive vectors of a tropical curve to a finite set.
\begin{corollary}
\label{cor:degree}
For a curve in $(\CC^*)^n$ of degree $d$, the absolute values of integers appearing in the primitive ray directions of its tropical curve is bounded from above by $d$.
\end{corollary}

\section{Implementation}\label{section:implementation}

We implement our algorithms in a Macaulay2 package using three other software packages for homotopy tracking: PHCpack~\cite{PHCpack}, Bertini~\cite{Bertini-book}, and NAG4M2~\cite{Leykin:NAG4M2}.
All three of them have exclusive features that are utilized by this project, which would be hard to accomplish without convenient Macaulay2 interfaces for the first two, \cite{GPV:PHCpackM2} and \cite{BGLR:BertiniM2}.

We assume that $I = \ideal{F} \subset \CC[\bx] = \CC[x_1,\ldots,x_n]$ where $F$ is a set of $n-1$ generators defining a curve in the torus $(\CC^*)^n$. What to do if $I$ is not a complete intersection in the torus is briefly discussed in Remarks~\ref{remark:C(s)} and \ref{remark:numerical-multiplicity}.

Note that we do not assume that $\dim I = 1$. In fact, in all of our nontrivial examples $V(I)\subset\CC^n$ is a higher dimensional variety with positive dimensional components in the coordinate hyperplanes.

As outlined in the beginning of~\S\ref{section:main} our approach has three parts. 

\subsection{Computing candidate rays} 
Our first task is to find the points $V(I_{\ba,C})\cap (\CC^*)^n$ where $I_{\ba,C} = \langle \bx^\ba - C\rangle$ as in (\ref{equation:I_aC}).  This can be done efficiently with polyhedral homotopies implemented in PHCpack~\cite{PHCpack}. 

Then we track the found points along the homotopy induced by $C(s)$ of Remark~\ref{remark:C(s)} to find the corresponding points of $V(I_{\ba,C'})\cap (\CC^*)^n$. This step can be accomplished with any of the three homotopy trackers. In practice, as the magnitudes of the solution coordinates grow (or approach 0), a homotopy tracker may give up on some paths: if the norm of an approximate solution exceeds a heuristically set threshold, then the corresponding path is truncated; the same happens if the path is judged as passing too closely near a singularity.  

The packages NAG4M2 and PHCpack perform much faster on this task than Bertini, although Bertini tends to give up on fewer paths, since it adapts precision according to numerical conditioning, while the other two operate with fixed (machine) precision.  

One way or another, the described heuristic procedure results in a set of difference vectors (as in Figure~\ref{fig:amoeba}), which are then converted into primitive integer vectors as described in~\S\ref{section:find}.
 
\subsection{Computing multiplicities}

Consider the branched cover of the complex plane where 
\begin{itemize}
\item the total space is the curve $V(F)\subset(\CC^*)^{n+1}$ where $F$ is the generating set in Lemma~\ref{lemma:main}, and
\item the covering map is the projection to the coordinate $t$.  
\end{itemize}
It is useful to see the path $a(s)\in\CC$, $s\in[0,1]$, of Remark~\ref{remark:numerical-multiplicity} as being embedded in the base space.  This is depicted in Figure~\ref{fig:base-space} where the ramification points are shown as red dots. 
As long as $a(s)$ does not go through the (finitely many) ramification points for $s\in[0,1)$, the lifted path in the total space is regular with a possible exception of the end $s=1$.   

\begin{figure}[h]
\includegraphics[scale=1.2]{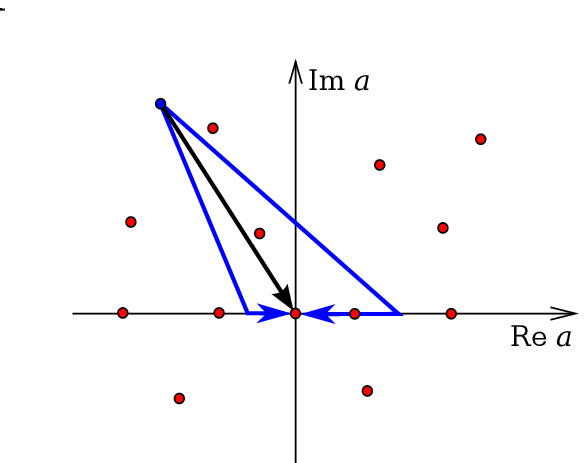}
\caption{Computing ray multiplicity: several ways to approach the origin.}
\label{fig:base-space}
\end{figure}

Implementing the homotopy induced by $a(t)$, we instruct the homotopy tracking software to be extremely conservative: for most nontrivial examples the target points are not only singular, but are also not isolated. In most problematic of our examples, the curve that is being tracked intersects some component of a much higher dimension at the boundary of the torus.    

\begin{remark} When the input ideal $I$ is defined over $\QQ$ the points in the torus to which our homotopy converges are solutions to $\tinit_\omega(I)$, which is again defined over $\QQ$. 

Not only can we double check that solutions are either real or come in conjugate pairs, but it is also possible to recover $\tinit_\omega(I)$ from these approximate solutions, given enough accuracy and a known (or assumed) bound on the coefficients of generators of $\tinit_\omega(I)$. 
\qed
\end{remark}

When the input is defined over $\RR$, our experiments suggest that instead of the path $a(s)$  approaching $0$ along a segment as suggested in Remark~\ref{remark:numerical-multiplicity}, it is more robust to break the path into two segments:
\begin{enumerate}
\item $a_1(s)$ approaching $a_1(1)=\varepsilon$, where $\varepsilon\in\RR$ is a small nonzero number, and
\item $a_2(s)$ starting at $a_2(0)=a_1(1)$ and approaching $a_2(1)=0$ along the real axis. 
\end{enumerate}
Note that in the second case (see Figure~\ref{fig:base-space}) the segment connecting $\varepsilon$ and the origin may contain a ramification point, therefore, some paths may be lost. Choosing $\varepsilon$
 ``sufficiently close'' to the origin adds yet another heuristic layer to our procedure.

\subsection{Verifying completeness}
Computing the degree of the curve in the torus can be done by slicing with a random hyperplane and approximating the resulting points via polyhedral homotopy of PHCpack (computing their multiplicities numerically if not regular). 

The non-numerical computations described in~\S\ref{section:degree} are carried out in Macaulay2.

\section{Example of the $A$-polynomial of a knot} \label{section:knots}

We tested our methods by computing tropical curves of $A$-polynomials of some knots. For each knot, one can associate a polynomial in two variables called the {\em $A$-polynomial}. The boundary slopes of the Newton polygon of the $A$-polynomial are {\em boundary slopes of incompressible surfaces} in the knot complement \cite{CCGLS}. Computation of these boundary slopes is of interest in the knot theory community. See, for example, \cite{GvdV}.

The computation of the $A$-polynomial can be reduced to an elimination problem.  We can get a set of polynomial equations that cut out a curve in a complex algebraic torus, for example, with the software SnapPy \cite{SnapPy}.  The desired plane curve defined by the $A$-polynomial is the image of this curve under a monomial map, also computed by SnapPy.  However, this set of elimination problems is very challenging.  

\begin{example}[Knot $8_1$] \label{example:8-1}
The algebraic curve of interest is defined by the ideal 
\begin{eqnarray*}
I = \langle&{z}_{1}+{w}_{1}-1,\ {z}_{2}+{w}_{2}-1,\ {z}_{3}+{w}_{3}-1,\ {z}_{4}+{w}_{4}-1,\ {z}_{5}+{w}_{5}-1,\ &\\
           &-{z}_{2} {z}_{4} {w}_{1} {w}_{5}+{w}_{2} {w}_{4},\ {z}_{2} {z}_{4} {z}_{5}^{2} {w}_{1}^{2}-{z}_{1}^{2} {w}_{2} {w}_{3} {w}_{4} {w}_{5},\ -{z}_{3}^{2} {w}_{1}+{w}_{3}^{2},\ -{z}_{2} {z}_{4} {z}_{5}^{2}+{w}_{5}^{2}& \rangle
\end{eqnarray*}
of $R=\QQ[z_1, \ldots, z_5, w_1, \ldots, w_5]$, a polynomial ring in 10 variables.

In the first stage of the algorithm 20 rays are deemed as candidates for parts of the tropicalization of the curve. The 8 primitive vectors below span the rays ``passing'' the second stage, i.e.\ their computed multiplicities are non-zero:
\[
\begin{array}{cc}
\text{multiplicity} & \text{ray}\\
3 &  (0,1,0,-1,1,0,1,0,0,1)\\ 4 &  (-1,1,0,1,-1,0,1,0,1,0)\\ 3 &  (0,-1,0,1,1,0,0,0,1,1)\\ 1 &  (0,0,0,-2,0,-4,-7,-2,0,-1)\\ 1 &  (0,-2,0,0,0,-4,0,-2,-7,-1)\\ 2 &  (2,-2,-1,0,0,2,0,0,-1,-1)\\ 2 &  (2,0,-1,-2,0,2,-1,0,0,-1)\\ 2 &  (-2,1,2,1,-1,0,1,2,1,0)
\end{array}
\]
One can check that the above rays, with multiplicities, sum up to the zero vector, so they satisfy the balancing condition. Moreover, the degree of the algebraic curve is computed numerically to be $22$, coinciding with the degree of the computed tropical curve. An independent Gfan computation confirms the list of rays.

One should point out that the points witnessing the listed multiplicities, which approximate the leading coefficients of the Puiseux series solutions corresponding to the ray, have coordinates that are close to algebraic numbers. For instance, the three points obtained for the first ray in the above list,
\[
\begin{array}{l}
(-2.24698, -1, .356896, -1, -.307979, 3.24698, 1, .643104, 1, .307979),\\
(.801938, -1, .692021, -1, -5.04892, .198062, 1, .307979, 1, 5.04892),\\
(-.554958, -1, -4.04892, -1, -.643104, 1.55496, 1, 5.04892, 1, .643104),
\end{array}
\] 
have the first coordinates satisfying $x^3+2x^2-x-1=0$ approximately. We can refine the approximation to an arbitrary precision.

The image of the curve under a monomial map is the one defined by the $A$-polynomial of the $8_1$ knot.  The tropicalization of the monomial map is a linear map given by the following matrix:
 $$\bgroup\begin{pmatrix}0&
     0&
     0&
     1&
     1&
     0&
     0&
     0&
     0&
     {-1}\\
     {-2}&
     1&
     0&
     3&
     4&
     2&
     {-3}&
     {-1}&
     1&
     {-3}\\
     \end{pmatrix}\egroup.$$
We can verify that the dual Newton polygon of the image plane curve has boundary slopes $-12$, $0$, and $4$, and that it coincides with Culler's computations~\cite{Culler}. 
\qed
\end{example}

\section{Future Directions}

We have presented a numerical method for computing a tropical curve from a generating set of its ideal.  Here we present some ideas for future research projects and possible applications.  Many details remain to be worked out carefully.

\subsection{Certification}\label{subsec:certification}

There is an established technique for certification in numerical AG, following from Smale's $\alpha$-theory, but it applies only to {\em regular} approximate zeros of a square system of polynomial equations. In the computation of the ray multiplicity described in~\S\ref{section:check}, if the point on the tracked homotopy path at $a=0$ is regular for the system of equations defining the path, then certification with this technique is possible.

However, the target solution of that homotopy may be singular.  Moreover, in many our examples it is not even isolated: there is a high dimensional component in the boundary of $(\CC^*)^n$ intersecting the curve containing the path at the target point. There are two directions that we see worth exploring:
\begin{itemize}
\item There is a potential for a hybrid symbolic-numerical technique: if the input is defined over $\QQ$, the the coordinates of the solutions that we obtain approximate algebraic numbers (see  Example~\ref{example:8-1}) and we can hope to recover generators of $\tinit_\omega(I)$ for the ray $\omega$. Assuming that the methods for symbolically computing $\tinit_\omega(I)$ are practically infeasible, can one still check that the initial ideal recovered from the numerical data is correct?
\item The possible component at $t=0$ that makes a target point singular could be in theory ``saturated out''. Assuming that computing $I:t^\infty$ is practically infeasible, is it possible to ``saturate numerically''? Here one should, perhaps, not be looking for rigorous certification but rather for a method that would improve the numerical approximation of the target points. 
\end{itemize}

\subsection{Toward Hybrid Methods}

At present, the leading method for computing tropical varieties from polynomials, implemented in Gfan \cite{Gfan,BJSST}, works by finding a cone in the tropical variety and traversing the tropical variety as a subfan of the Gr\"obner fan.  The numerical method fits in well with this approach and may help reduce the number of Gr\"obner basis computations.  To find a starting cone in the tropical variety, we can first cut the ideal down to a one dimensional one, and then find a ray in the tropical curve using our numerical methods.  The Gr\"obner cone of the original ideal containing this ray is a starting cone.  During fan traversal, we need to compute the link of the tropical variety at a codimension one cone.  Modulo the lineality space, the link is a tropical curve, so we can use the numerical oracle again. 

\subsection{Toward Higher Dimensions}

For simplicity let us assume that the variety is equidimensional.  We can try to compute various tropical curves lying in the tropical variety by slicing the variety with binomials.  We then need to develop polyhedral algorithms for patching the slices together.  Moreover, if we want to take affine slices away from the origin, we first need to extend our methods to non-constant coefficient tropical curves.

We can also reduce the problem to the case of hypersurfaces.
By the Hept--Theobald Theorem~\cite{zbMATH05567230} we can find a tropical basis of a $d$ dimensional tropical variety in $\RR^n$ by linearly projecting it onto some $d+1$ dimensional linear subspaces so that the images are hypersurfaces.  This corresponds to projecting the original variety in $(\CC^*)^n$ onto some $d+1$ dimensional subtori via monomial maps. Since we can compute the witness sets of projections numerically, we can use the methods of Hauenstein and Sottile to compute the dual Newton polytopes~\cite{HauensteinSottile}.  Alternatively, if we can compute non-constant coefficient tropical curves, then we can slice a tropical hypersurfaces down to curves, obtaining $2$-dimensional faces of the dual Newton polytope that can be fitted together to get all vertices of the polytope.  The tropical hypersurfaces of the tropical basis determine the support of the tropical variety.

Our method for computing multiplicities, however, works for tropical varieties of arbitrary (pure) dimension, as we can slice any variety down to zero-dimension and count convergent homotopy paths as in~\S\ref{section:check}.



\subsection{Applications to knot theory}
As we saw above, the knot examples provide a good family of interesting tropical curves in high ambient dimension that project to planar tropical curves, the tropicalizations of curves defined by the $A$-polynomials of knots.  

Computing that tropicalization (without knowing the $A$-polynomial) produces the {\em boundary slopes} of its Newton polygon, which are knot invariants that are already interesting. Moreover, since we are able to compute the multiplicities of the rays of the curve, the Newton polygon can be recovered: having better estimates of the monomial support may push the boundary of the numerical interpolation technique for computing $A$-polynomials developed by Culler {\em et al}~\cite{Culler}. In addition, the (non-planar) tropical curve that we compute provides even finer information about the knot that may have an interesting topological explanation. 

It will be worthwhile to explore the vast territory not yet known to knot theorists with hybrid methods for computing tropical curves.

\bibliographystyle{amsplain}
\bibliography{mybib} 
 
\end{document}